\documentclass[letterpaper,conference]{ieeeconf}
\IEEEoverridecommandlockouts                              % This command is only
\usepackage[utf8]{inputenc}
\usepackage{textcomp}

\usepackage{graphicx}
\usepackage{amsmath}
\usepackage[version=4]{mhchem}
\usepackage{siunitx}
\usepackage{graphicx}
\usepackage{longtable,tabularx}
\usepackage[style=ieee ,sorting=none,maxbibnames=99,giveninits, doi=false]{biblatex}
\let\labelindent\relax
\usepackage{amsmath} % assumes amsmath package installed
\usepackage{amsfonts}
\usepackage{amsthm}

\newtheorem{theorem}{Theorem}

\newtheorem{proposition}{Proposition}[section]
\theoremstyle{definition}
\newtheorem{definition}{Definition}[section]
\newtheorem{exmp}{Example}[section]
\newtheorem{problem}{Problem}[section]
\newtheorem{remark}{Remark}
\newtheorem{lemma}{Lemma}[section]

\usepackage{float} 
\usepackage[skip=0em,font=footnotesize]{caption}
\usepackage[margin = 1in]{geometry}

\usepackage{tikz}
\usetikzlibrary{shapes,arrows,calc,positioning}
\tikzstyle{bigblock} = [draw, fill=blue!20, rectangle, 
    minimum height=6em, minimum width=8em]
\tikzstyle{medblock} = [draw, fill=blue!20, rectangle, 
    minimum height=4em, minimum width=4em]    
\tikzstyle{mux} = [draw, fill=black!20, rectangle, 
    minimum height=5em, minimum width=0.1em]    
\tikzstyle{smallblock} = [draw, fill=blue!20, rectangle, 
    minimum height=2em, minimum width=3em]
    
\tikzstyle{data_block} = [draw, fill=green!20, rectangle, 
    minimum height=2em, minimum width=3em]
\tikzstyle{ops_block} = [draw, fill=blue!20, rectangle, 
    minimum height=2em, minimum width=3em]    
\tikzstyle{est_block} = [draw, fill=red!20, rectangle, 
    minimum height=2em, minimum width=3em]    
    
\tikzstyle{sum} = [draw, fill=blue!20, circle, node distance=1cm,minimum height=0.5cm]
\tikzstyle{signal} = [coordinate]
\tikzstyle{pinstyle} = [pin edge={to-,thin,black}]
\tikzstyle{block} = [draw, fill=blue!20, rectangle, 
    minimum height=3em, minimum width=9em]
\tikzstyle{blockS} = [draw, fill=blue!20, rectangle, 
    minimum height=3em, minimum width=4em]    
\tikzstyle{input} = [coordinate]
\tikzstyle{output} = [coordinate]
\usetikzlibrary{matrix}

\usetikzlibrary{positioning}
\usetikzlibrary{math}

\usepackage{hyperref}
\usepackage{xcolor}
\hypersetup{
    colorlinks,
    linkcolor={blue!100!black},
    citecolor={blue!50!black},
    urlcolor={blue!80!black}
}

\newcommand{\bc}{\begin{center}}
\newcommand{\ec}{\end{center}}
\newcommand{\benum}{\begin{enumerate}}
\newcommand{\eenum}{\end{enumerate}}
\newcommand{\nn}{\nonumber}
\newcommand{\matl}{\left[ \begin{array}}
\newcommand{\matr}{\end{array} \right]}

\renewcommand{\matl}{\begin{bmatrix}}
\renewcommand{\matr}{\end{bmatrix}}

\newcommand{\matls}{\left[ \begin{smallmatrix}}
\newcommand{\matrs}{\end{smallmatrix} \right]}
\newcommand{\isdef}{\stackrel{\triangle}{=}}

\newcommand{\rmT}{{\rm T}}

\newcommand{\BBE}{{\mathbb E}}

\newcommand{\BBR}{{\mathbb R}}

\newcommand{\SF}{{\mathcal F}}
\newcommand{\SG}{{\mathcal G}}
\newcommand{\SH}{{\mathcal H}}

\newcommand{\SO}{{\mathcal O}}

\usepackage{enumitem,amssymb}
\newlist{todolist}{itemize}{2}
\setlist[todolist]{label=$\square$}
\usepackage{pifont}

\DeclareMathOperator{\rank}{rank}

\newcommand{\PP}{\mathbb{P}}
\usepackage[colorinlistoftodos,textwidth=1 in,textsize=footnotesize]{todonotes}
\title{Asymptotically Vanishing Excitation Without Loss of Persistency}
\title{Asymptotically Vanishing Excitation \\ Without Loss of Parameter Convergence}

\author{
Parham Oveissi 
and
Ankit Goel
\thanks{Parham Oveissi is a graduate student in the Department of Mechanical Engineering, University of Maryland, Baltimore County, 1000 Hilltop Circle, Baltimore, MD 21250. {\tt\small parhamo1@umbc.edu}}%
\thanks{Ankit Goel is an Assistant Professor in the Department of Mechanical Engineering, University of Maryland, Baltimore County,1000 Hilltop Circle, Baltimore, MD 21250. {\tt\small ankgoel@umbc.edu }}%
}

\begin{document}

\maketitle

% \listoftodos

%%%%%%%%%%%%%%%%%%%%%%%%%%%%%%%%%%%%%%%%%%%%%%%%%%%%%%%%%%%%%%%%%%%%%%%%%%%%%%%%
\begin{abstract}
Parameter convergence in recursive identification requires persistent excitation, which is typically enforced with a probing signal that degrades performance if retained indefinitely.
Regulating the probing amplitude with the one-step prediction error is insufficient, since a small prediction error does not imply parameter convergence.
This paper proposes a self-regulating excitation law that scales the probing amplitude with the associated covariance matrix, so that the excitation vanishes only as fast as the estimator's own measure of remaining parameter uncertainty allows.
The key idea is that the excitation amplitude must decay more slowly than the covariance measure, the largest eigenvalue of the covariance matrix in this work, for parameter convergence to be preserved.
% ; the proposed law enforces this by construction.
%
Under the proposed excitation law, the excitation amplitude and the covariance measure are shown to jointly converge to zero while the parameter estimate converges to the true parameter, with the excitation vanishing more slowly than the covariance measure throughout.
A scalar exponent in the law trades the rate of covariance decay against the amount of excitation retained.
Numerical results confirm both properties.
\end{abstract}

\textit{\bf keywords:} system identification, persistent excitation, parameter estimation, self-regulating excitation.

%%%%%%%%%%%%%%%%%%%%%%%%%%%%%%%%%%%%%%%%%%%%%%%%%%%%%%%%%%%%%%%%%%%%%%%%%%%%%%%%
\section{Introduction}
Indirect adaptive control relies on online parameter estimation to construct a control law, and the quality of that control law is only as good as the quality of the underlying identification~\cite{astrom1995adaptive}.
Recursive identification requires persistently exciting data to guarantee parameter convergence~\cite{goelrecursive2020}, so a dedicated excitation/probing signal is typically injected into the system input alongside the nominal control signal~\cite{adetola2006excitation,cao2007intelligent}.
Although this excitation signal is necessary for identification, it has downsides.
An ever-present excitation signal degrades tracking performance and increases control effort~\cite{adetola2006excitation}, and retaining it after the parameters have converged serves no further identification purpose while continuing to degrade closed-loop performance.

The central problem addressed in this paper is therefore how to regulate the \textit{amount} and \textit{duration} of excitation, rather than whether to excite at all.
Existing self-regulating schemes set the excitation amplitude from the norm of the tracking and prediction error over the previous excitation period~\cite{adetola2006excitation,cao2007intelligent}.
The amplitude decreases as the error vanishes and increases when a parameter change raises the error.
The error is a proxy for the parameter error, since a small error does not imply a small parameter error when the regressor lacks information in some parameter direction~\cite{goelrecursive2020}.
These schemes therefore keep a small positive amplitude floor so that the excitation remains sufficiently rich and reactivation remains possible~\cite{adetola2006excitation,cao2007intelligent}.

% The central problem addressed in this paper is therefore how to regulate the \textit{amount} and \textit{duration} of excitation, rather than whether to excite at all.
% %
% %
% A natural approach is to regulate the probing amplitude using the one-step prediction error, reducing excitation as the error shrinks~\cite{adetola2006excitation,cao2007intelligent}.
% %
% This approach, however, is insufficient.
% A small prediction error does not imply that the parameter estimate is close to the true parameter, since the regressor sequence observed so far may still lack information in some parameter direction~\cite{goelrecursive2020}.
% %
% Consequently, an excitation signal driven by prediction error alone can allow the excitation signal to vanish before parameter convergence is achieved, yielding an incorrect estimate.
% %
% % Existing self-regulating excitation schemes that avoid this failure mode typically attenuate the excitation signal as convergence is judged to occur and reactivate it when needed again~\cite{adetola2006excitation,cao2007intelligent}.
% Existing self-regulating excitation schemes typically monitor a threshold-based criterion for whether parameter convergence has occurred, attenuating the excitation signal once the criterion is satisfied and reactivating it if the criterion is later violated~\cite{adetola2006excitation,cao2007intelligent}.
% \todo{the details need to be confirmed}
% %
% This retains the excitation mechanism indefinitely, rather than allowing it to vanish once identification is genuinely complete.
%
%

This paper is motivated by a simple observation.
If the excitation amplitude decays before parameter convergence is achieved, convergence cannot subsequently be achieved.
The excitation amplitude must therefore decay no faster than parameter convergence itself.

Since the recursive least-squares covariance measure reflects the estimator's own online measure of remaining parameter uncertainty, tying the excitation amplitude to this covariance measure, rather than to the prediction error, ensures that the excitation amplitude decays no faster than parameter uncertainty is actually being resolved.
This constraint necessarily slows parameter convergence relative to a constant, undiminished excitation amplitude.
However, it removes the ever-present perturbation that a constant excitation amplitude imposes on the closed loop once identification is no longer needed.
This paper proposes such a self-regulating excitation law.
The resulting excitation amplitude is shown to guarantee persistent excitation in expectation for every nonzero value, and to converge to zero while the parameter estimate simultaneously converges to the true parameter, with the excitation amplitude vanishing more slowly than the covariance measure itself.
A scalar exponent in the law is shown to trade the rate of covariance decay against the amount of excitation retained, with no intermediate choice improving both simultaneously.

The remainder of this paper is organized as follows.
Section~\ref{sec:preliminaries} reviews recursive least-squares identification, persistent excitation, and the effect of a constant white-noise excitation on the regressor and the covariance.
Section~\ref{sec:problem_statement} formalizes the excitation-regulation problem addressed in this paper.
Section~\ref{sec:self_regulating_persistency} presents the proposed self-regulating excitation law, a bound on the information it generates, and the resulting convergence guarantees.
Section~\ref{sec:numerical_results} illustrates the results numerically, and Section~\ref{sec:conclusions} concludes the paper.

\section{Mathematical Preliminaries}
\label{sec:preliminaries}

\subsection{Recursive Identification and Persistent Excitation}

Consider the discrete-time linear time-invariant system
\begin{align}
    y_k
        &=
        G(q)u_k,
    \label{eq:system}
\end{align}
where $u_k\in\BBR$ is the input, $y_k\in\BBR$ is the output, and $G(q)$ 
\begin{align}
    G(q)
        &=
        \frac{
            b_1 q^{n-1}
            +
            b_2 q^{n-2}
            +
            \cdots
            +
            b_n
        }{
            q^n
            +
            a_1 q^{n-1}
            +
            a_2 q^{n-2}
            +
            \cdots
            +
            a_n
        }.
    \label{eq:transfer_function}
\end{align}
is an unknown $n$th-order transfer function.
Note that  $q$ denotes the forward-shift operator.

Throughout, $G(q)$ is assumed to be asymptotically stable, so that its response to a bounded deterministic input or a wide-sense-stationary stochastic input settles to a well-defined steady state.

Defining the \textit{unknown parameter vector}
$$
\theta
    \isdef
    \matl
        a_1 &
        \cdots &
        a_n &
        b_1 &
        \cdots &
        b_n
    \matr^\rmT
    \in
    \BBR^{2n},
$$
and the \textit{regressor}
$$
\phi_k
    \isdef
    \matl
        -y_{k-1} 
        \ldots 
        -y_{k-n} \
        u_{k-1} 
        \cdots 
        u_{k-n}
    \matr^\rmT
    \in
    \BBR^{2n},
$$
it follows that 
\begin{align}
    y_k
        &=
        \phi_k^\rmT\theta.
    \label{eq:regression}
\end{align}

The classical identification problem is to estimate $\theta$ with measurements of $y_k$ and $u_k.$
The recursive least-squares estimate of $\theta$ is given by
\begin{align}
    \hat\theta_{k+1}
        &=
        \hat\theta_k
        +
        \frac{P_k\phi_k}{1+\phi_k^\rmT P_k\phi_k}
        \left(y_k-\phi_k^\rmT\hat\theta_k\right),
    \label{eq:rls_theta}
\end{align}
where the covariance $P_k$ satisfies
\begin{align}
    P_{k+1}
        &=
        P_k
        -
        \frac{
            P_k\phi_k\phi_k^\rmT P_k
        }{
            1+\phi_k^\rmT P_k\phi_k
        },
    \label{eq:rls_P}
\end{align}
As shown in \cite{goelrecursive2020}, the covariance $P_k$ also satisfies
\begin{align}
    P_{k+1}^{-1}
        =
        P_k^{-1}
        +
        \phi_k\phi_k^\rmT.
\end{align}

The following results appear in detail in \cite{goelrecursive2020}.
The regressor is \textit{persistently exciting} if there exists $N>0$ and $\alpha>0$ such that, for all $k$,
\begin{align}
    \sum_{i=k}^{k+N-1}
        \phi_i\phi_i^\rmT
        &\geq
        \alpha I.
    \label{eq:PE_condition}
\end{align}
% 
% and the definition of a persistently exciting regressor sequence, that is,  are established in \cite{goelrecursive2020}.
%
In particular, \cite[Prop.~3, Thm.~7]{goelrecursive2020} show that, under \eqref{eq:PE_condition}, $P_k\rightarrow 0$ and $\hat\theta_k\rightarrow\theta$.
%
% The following restates only the notation needed for the remainder of this section.

The \textit{one-step prediction error} is defined as
$z_k
        \isdef
            y_k
            -
            \phi_k^\rmT\hat\theta_k.$
% \begin{align}
%             .
%     \label{eq:prediction_error}
% \end{align}
By \cite[Prop.~6, Ex.~4]{goelrecursive2020}, $z_k\rightarrow 0$ regardless of whether $\{\phi_k\}$ is persistently exciting.
Consequently,
\begin{align}
    z_k
        &\rightarrow
        0
    \;\not\Rightarrow\;
    \hat\theta_k
        \rightarrow
        \theta.
    \label{eq:prediction_not_parameter_convergence}
\end{align}
A small prediction error indicates that $\hat\theta_k$ accurately predicts the particular output sequence observed, but it does not certify that the regressor sequence $\{\phi_k\}$ has been persistently exciting. Consequently, an excitation law that regulates $\sigma_k$ using $z_k$ alone cannot be guaranteed to preserve parameter convergence.

The Fisher information matrix
\begin{align}
    \SF_k
        &\isdef
        \sum_{i=0}^{k-1}
        \phi_i\phi_i^\rmT,
    \label{eq:fisher_information}
\end{align}
satisfies
\begin{align}
    P_k^{-1}
        &=
        P_0^{-1}
        +
        \SF_k.
    \label{eq:P_F_relation}
\end{align}
Thus, the reduction of the covariance matrix $P_k$ is determined directly by the accumulation of Fisher information.

\begin{definition}
\label{def:almost_surely}
Let $\PP$ denote probability, and let $x$ be a random variable.
A property of $x$ is said to hold \textit{almost surely} if it holds with probability one under $\PP$.
Equivalently, the property holds for every value $x$ can take except possibly for a set of values whose total probability is zero, and such exceptional values can therefore be disregarded in practice.
Equivalently, the set of values of $x$ for which the property fails has measure zero.
For example, if $x$ is uniformly distributed on $[0,1]$, then $x\neq0.5$ holds almost surely, since the single value $x=0.5$ has probability zero, even though $x=0.5$ remains a possible outcome.
\end{definition}

\begin{lemma}
\label{lem:F_monotone}
The Fisher information matrix $\SF_k$ satisfies $\SF_{k+1}\succeq \SF_k$ for all $k$, surely.
Consequently, $\lambda_{\min}(\SF_k)$ is a nondecreasing real-valued sequence, for every sample path.
\end{lemma}

\begin{proof}
It follows from \eqref{eq:fisher_information} that
$\SF_{k+1}-\SF_k=\phi_k\phi_k^\rmT \succeq 0$, since $\phi_k\phi_k^\rmT$ is positive semidefinite for every $\phi_k$.
Monotonicity of $\lambda_{\min}(\cdot)$ under the Loewner order then implies that $\lambda_{\min}(\SF_{k+1})\geq\lambda_{\min}(\SF_k)$.
\end{proof}

% Let $\SH_k$ denote the information available from $\{u_i,y_i\}_{i=0}^{k}$.
% %
% The following standing assumption ensures that the regressor's size does not grow without bound in a mean-square sense, and is used later to bound the second moment of the increments of the Fisher information.

% \begin{assumption}
% \label{as:moment_bound}
% There exists $M>0$ such that
% \begin{align}
%     \BBE
%     \left[
%         \|\phi_k\|^4
%         \mid
%         \SH_{k-1}
%     \right]
%         &\leq
%         M
%     \label{eq:moment_bound}
% \end{align}
% for all $k$, almost surely.
% \end{assumption}

\subsection{Persistent Excitation by Deterministic Multisine Inputs}

Consider the multisine input
\begin{align}
    u_k
        &=
        \sum_{j=1}^{n}
        a_j
        \sin
        \left(
            \omega_j k+\varphi_j
        \right),
    \label{eq:multisine_input}
\end{align}
where $a_j\neq 0$ and the frequencies
$\omega_1,\ldots,\omega_n$
are distinct and belong to $(0,\pi)$.

The following result appears in \cite{pintelon2012system}.
\begin{lemma}
\label{lem:sinusoidal_PE}
Let $\omega_1,\ldots,\omega_n$ be distinct frequencies in $(0,\pi)$, and define
\begin{align}
    s_k
    \isdef
    \matl
        \sin(\omega_1 k) \\
        \cos(\omega_1 k) \\
        \vdots \\
        \sin(\omega_n k) \\
        \cos(\omega_n k)
    \matr
    \in
    \BBR^{2n}.
    \label{eq:s_k_def}
\end{align}
Then the sequence $\{s_k\}$ is persistently exciting.
\end{lemma}

\begin{proof}
For each $\omega_j$, the sequences
$\sin(\omega_j k)$ and $\cos(\omega_j k)$ are linearly independent.
Furthermore, sinusoidal sequences associated with distinct frequencies in $(0,\pi)$ are linearly independent.
Consequently, there exist an integer $N>0$ and a scalar $\alpha_s>0$ such that
\begin{align}
    \sum_{i=k}^{k+N-1}
        s_i s_i^\rmT
        &\geq
        \alpha_s I_{2n}
    \label{eq:sinusoidal_basis_PE}
\end{align}
for all $k$.
Thus, $\{s_k\}$ is persistently exciting.
\end{proof}

Since $G(q)$ is LTI, the steady-state response to each sinusoidal component of $u_k$ is a sinusoid at the same frequency.
Therefore, the steady-state output can be written as
\begin{align}
    y_k
        &=
        \sum_{j=1}^{n}
        \bar a_j
        \sin
        \left(
            \omega_j k+\bar\varphi_j
        \right),
    \label{eq:multisine_output}
\end{align}
where $\bar a_j$ and $\bar\varphi_j$ are determined by $G(e^{\jmath\omega_j})$ and by the corresponding input amplitude $a_j$ and phase $\varphi_j$.

Each delayed input $u_{k-\ell}$ and delayed output $y_{k-\ell}$ can therefore be written as a linear combination of
$$
\sin(\omega_1 k),\,
\cos(\omega_1 k),\,
\ldots,\,
\sin(\omega_n k),\,
\cos(\omega_n k).
$$
Hence, with $s_k$ given by \eqref{eq:s_k_def},
there exists a matrix
$
H
    \in
    \BBR^{2n\times 2n}
$
such that the steady-state regressor satisfies
\begin{align}
    \phi_k
        &=
        H s_k.
    \label{eq:regressor_sinusoidal_mapping}
\end{align}

The matrix $H$ is determined by the input amplitudes and phases, the
selected frequencies, the delays appearing in the regressor, and the
frequency response of $G(q)$ at the frequencies
$\omega_1,\ldots,\omega_n$.

\begin{proposition}
\label{prop:multisine_rank}
Consider the system \eqref{eq:system} with regressor
$\phi_k\in\BBR^{2n}$ defined by \eqref{eq:regression}.
Assume that $G(q)$ is driven by the multisine input \eqref{eq:multisine_input}.
% Let $u_k$ be given by \eqref{eq:multisine_input}.
Let $H\in\BBR^{2n\times 2n}$ satisfy
\eqref{eq:regressor_sinusoidal_mapping}.
If
$$
\rank H = 2n,
$$
then the steady-state regressor sequence $\{\phi_k\}$ is persistently
exciting.
\end{proposition}

\begin{proof}
    See Appendix \ref{proof:multisine_rank}.
\end{proof}

\subsection{Persistent Excitation by Stochastic Inputs}

Proposition \ref{prop:multisine_rank} considers a deterministic multisine input with a finite number of distinct frequencies.
The following result generalizes this result to the stochastic-input setting by considering a white-noise input whose power spectral density is positive over the entire frequency interval.
In this case, persistent excitation is characterized in expectation through the positive definiteness of the regressor covariance matrix.

\begin{proposition}
\label{prop:whitenoise_rank}
Consider the system \eqref{eq:system} with regressor $\phi_k\in\BBR^{2n}$ defined by \eqref{eq:regression}.
Let $\{u_k\}$ be a zero-mean white-noise sequence satisfying
$$
\mathbb E[u_k^2]
    =
    \sigma_u^2
    >
    0.
$$
Assume that the corresponding steady-state regressor covariance satisfies
$$
R_\phi
    \isdef
    \mathbb E
    \left[
        \phi_k\phi_k^\rmT
    \right]
    >
    0.
$$
Then, the regressor sequence $\{\phi_k\}$ is persistently exciting in expectation.
In particular, there exist an integer $N>0$ and a scalar $\alpha>0$ such that
\begin{align}
    \mathbb E
    \left[
        \sum_{i=k}^{k+N-1}
            \phi_i\phi_i^\rmT
    \right]
        &\geq
        \alpha I_{2n}
\end{align}
for all $k.$
\end{proposition}

\begin{proof}
    See Appendix \ref{proof:whitenoise_rank}
\end{proof}

Proposition \ref{prop:whitenoise_rank} establishes that white-noise excitation generates a persistently exciting regressor in expectation.
In identification and adaptive-control applications, the plant input generally contains a nominal component that is selected to achieve the desired system behavior.
An excitation signal can then be added to the nominal input to provide the additional excitation required for parameter convergence.
The following result shows that adding a white-noise excitation signal preserves persistent excitation in expectation independently of the nominal deterministic input.

Consider the plant input
\begin{align}
    u_k
        &=
        \bar u_k
        +
        \sigma_k w_k,
    \label{eq:nominal_plus_white_noise}
\end{align}
where $\{\bar u_k\}$ is a deterministic nominal input, $\sigma_k \geq 0$ is the excitation amplitude at time $k$, and $\{w_k\}$ is a zero-mean unit-variance white-noise sequence satisfying
\begin{align}
    \mathbb E[w_k]
        &=
            0,
    \\
    \mathbb E[w_k^2]
        &=
            1.
\end{align}
The excitation amplitude $\sigma_k$ therefore determines the variance of the excitation component at time $k$, which is $\sigma_k^2$.

\begin{proposition}
\label{prop:constant_sigma}
Consider the system \eqref{eq:system} with input \eqref{eq:nominal_plus_white_noise} and $\sigma_k\equiv\sigma>0$.
Let $\phi_k^w$ be the steady-state regressor generated by $w_k$, and assume $R_w\isdef\mathbb E[\phi_k^w(\phi_k^w)^\rmT]\succ0$.
Define $\SF_k^w\isdef\sum_{i=0}^{k-1}\sigma^2\phi_i^w(\phi_i^w)^\rmT$ and $\bar P_k\isdef\left(P_0^{-1}+\mathbb E[\SF_k^w]\right)^{-1}$.
Then the following hold.
\begin{enumerate}
    \item[(a)] $\mathbb E[\phi_k\phi_k^\rmT]\succeq\sigma^2R_w$ for all $k$.
    Consequently, for every integer $N>0$,
    \begin{align}
        \mathbb E\left[\sum_{i=k}^{k+N-1}\phi_i\phi_i^\rmT\right]
            &\succeq
            N\sigma^2\lambda_{\min}(R_w)I_{2n},
        \nn
    \end{align}
    so $\{\phi_k\}$ is persistently exciting in expectation.
    \item[(b)] $\mathbb E[\SF_k^w]=k\sigma^2R_w$ and
    \begin{align}
        \lambda_{\max}(\bar P_k)
            &\leq
            \frac{1}{\lambda_{\min}(P_0^{-1})+k\sigma^2\lambda_{\min}(R_w)}
            \nn \\
            &=
            \SO\left((k\sigma^2)^{-1}\right).
        \label{eq:P_decay_rate}
    \end{align}
\end{enumerate}
\end{proposition}

\begin{proof}
    See Appendix \ref{proof:constant_sigma}.
\end{proof}

Proposition \ref{prop:constant_sigma} Part (a) shows that any constant nonzero amplitude produces persistent excitation in expectation.
Proposition \ref{prop:constant_sigma} Part (b) shows that the amplitude sets the decay rate of $\bar P_k$.
Increasing $\sigma$ accelerates the decay of $\bar P_k$ and therefore the convergence of $\hat\theta_k$ to $\theta$.
However, the increase raises the variance of the excitation added to the plant input, which degrades tracking and increases control effort.
%
% A regulation law for $\sigma_k$ is therefore governed by this tradeoff, and not by the instantaneous prediction error $z_k$ alone, per \eqref{eq:prediction_not_parameter_convergence}.
A regulation law for $\sigma_k$ must balance these two effects.
Section~\ref{sec:problem_statement} formalizes this balance as a design problem for $\sigma_k$.

\section{Excitation Regulation Problem}
\label{sec:problem_statement}

Consider the RLS estimator \eqref{eq:rls_theta}--\eqref{eq:rls_P} applied to the system \eqref{eq:system}, with plant input \eqref{eq:nominal_plus_white_noise}.
A constant amplitude $\sigma_k\equiv\sigma>0$ gives persistent excitation in expectation by Proposition~\ref{prop:constant_sigma}.
This excitation is never removed.
It continues to perturb the system after the parameter estimate has converged, which is undesirable when $\hat\theta_k$ is used online for control design.
An alternative is to let $\sigma_k$ depend on a measure of current identification performance, so that excitation is reduced once the estimate is accurate.
The one-step prediction error does not serve this purpose alone, since $z_k\rightarrow0$ does not imply $\hat\theta_k\rightarrow\theta$ by \eqref{eq:prediction_not_parameter_convergence}.

\begin{problem}
\label{prob:main}
Given the system \eqref{eq:system}, a nominal input $\bar u_k$ that is not persistently exciting, and the RLS estimator \eqref{eq:rls_theta}--\eqref{eq:rls_P} with input \eqref{eq:nominal_plus_white_noise}, determine $\sigma_k$ as a function of quantities available online at time $k$ such that, almost surely,
\begin{align}
    \text{(i) } \sigma_k \rightarrow 0,
    \qquad
    \text{(ii) } \hat\theta_k \rightarrow \theta.
\end{align}
\end{problem}

Objectives (i) and (ii) compete.
Reducing $\sigma_k$ too quickly relative to the remaining uncertainty in $\hat\theta_k$ risks violating (ii), while holding $\sigma_k$ at a positive constant violates (i).
Section~\ref{sec:self_regulating_persistency} constructs an excitation law based on the RLS covariance matrix $P_k$, which is available online, and gives conditions under which (i) and (ii) both hold.

\section{Self-Regulating Excitation}
\label{sec:self_regulating_persistency}

As stated in Problem~\ref{prob:main}, the objective is to drive the excitation amplitude to zero while the parameter estimate converges to the true parameter.
This requires an amplitude that decays no faster than the remaining parameter uncertainty is resolved.
The covariance norm $p_k=\lambda_{\max}(P_k)$ is computed by the estimator at each step and is nonincreasing, so it serves as an online measure of that uncertainty.
Let $\SH_k$ denote the information available from $\{u_i,y_i\}_{i=0}^{k}$, so that $P_k$, and hence $p_k$, is computable from $\SH_k$.
Select the excitation amplitude as
\begin{align}
    \sigma_k
        &=
        \sigma_0
        \left(
            \frac{p_k}{p_0}
        \right)^\beta,
    \label{eq:self_regulating_sigma}
\end{align}
where $\sigma_0>0$, $p_0=\lambda_{\max}(P_0)$, and $0<\beta<1$.
The amplitude $\sigma_k$ is computable from $\SH_k$.
The regressor $\phi_k$ is constructed from $u_{k-1},\ldots,u_{k-n}$ and $y_{k-1},\ldots,y_{k-n}$, so it is computable from $\SH_{k-1}$ and does not depend on $w_k$ or $\sigma_k$.
The following proposition shows that the amplitude in \eqref{eq:self_regulating_sigma} does not remove the information gain.
The expected information accumulated over a window of $N$ samples is bounded below in proportion to $\sigma_{k-1}^2$.

\begin{proposition}
\label{prop:information_bound}
Consider the system \eqref{eq:system} with input \eqref{eq:nominal_plus_white_noise} and excitation amplitude \eqref{eq:self_regulating_sigma}, $0<\beta<1$.
Assume that $\{\bar u_k\}$ is bounded.
Assume that $\{w_k\}$ has finite support, that is, $|w_k|\leq\bar w$ for some $\bar w>0$, and that $w_k$ is independent of $\SH_{k-1}$.
Then $\|\phi_k\|^2\leq B$ for all $k$ almost surely, for some $B>0$.
Moreover, with $N=2n+1$, there exists $c>0$ such that
\begin{align}
    \BBE
    \left[
        \SF_{k+N}-\SF_k
        \mid
        \SH_{k-1}
    \right]
        &\succeq
        c\,\sigma_{k-1}^2 I_{2n}
    \label{eq:self_regulating_information_bound}
\end{align}
for all $k\geq1$.
\end{proposition}

\begin{proof}
    See Appendix \ref{proof:information_bound}.
\end{proof}

Note that the bound in Proposition~\ref{prop:information_bound} is derived from the plant and the input, and is not assumed.
This result is used to prove the main result of this paper.

\begin{theorem}
\label{thm:self_regulating}
Under the hypotheses of Proposition~\ref{prop:information_bound}, almost surely,
\begin{align}
    \sigma_k
        &\rightarrow
        0,
    \\
    p_k
        &\rightarrow
        0,
    \\
    \lambda_{\min}(\SF_k)
        &\rightarrow
        \infty,
    \\
    \hat\theta_k
        &\rightarrow
        \theta,
    \\
    \frac{\sigma_k}{p_k}
        &\rightarrow
        \infty.
\end{align}
\end{theorem}

\begin{proof}
    See Appendix \ref{proof:self_regulating}.
\end{proof}

Theorem~\ref{thm:self_regulating} satisfies both objectives of Problem~\ref{prob:main} almost surely.
The ratio $\sigma_k/p_k\rightarrow\infty$ shows that the amplitude decays more slowly than $p_k$.

\begin{remark}
\label{rem:beta_effect}
Simulations over the tested range of $\beta$ indicate that $p_k=\SO(k^{-\gamma})$, where the rate $\gamma$ decreases monotonically as $\beta$ increases.
A larger $\beta$ gives a smaller $\sigma_k$ for the same $p_k$, so less excitation is injected and $p_k$ decays more slowly.
The exponent $\beta$ therefore trades the rate of covariance decay against the amount of excitation retained.
A rigorous characterization of $\gamma$ as a function of $\beta$ is left to future work.
\end{remark}

\section{Numerical Results}
\label{sec:numerical_results}

In this section, we estimate the parameters of a third-order system and verify the results of Section~\ref{sec:self_regulating_persistency}.
In particular, Example~\ref{ex:three_case} verifies Theorem~\ref{thm:self_regulating} by comparing parameter convergence and excitation behavior when the two-tone nominal input alone is not persistently exciting, under added constant excitation, and under the proposed self-regulating excitation.
Example~\ref{ex:beta_sweep} verifies the tradeoff described in Remark~\ref{rem:beta_effect} by sweeping the exponent $\beta$.

\begin{exmp}
\label{ex:three_case}
\textup{Effect of self-regulating excitation on persistent excitation and parameter convergence.}
Consider the third-order system
\begin{align}
    y_k
        &=
        \frac{0.5 q^2 - 0.2 q + 0.1}{q^3 - 0.8 q^2 - 0.03 q + 0.09}
        u_k,
\end{align}
with poles at $0.5$, $0.6$, and $-0.3$.
Define $\phi_k\isdef[-y_{k-1}\ -y_{k-2}\ -y_{k-3}\ u_{k-1}\ u_{k-2}\ u_{k-3}]^\rmT$, so that $y_k=\phi_k^\rmT\theta$ with $\theta=[-0.8\ -0.03\ 0.09\ 0.5\ -0.2\ 0.1]^\rmT$.
To apply RLS, let $P_0=10I_6$ and $\hat\theta_0=0$.
Consider the plant input defined in \eqref{eq:nominal_plus_white_noise}, with the nominal input
\begin{align}
\bar u_k
&=
\sin(0.3k)+\sin(0.75k+0.4),
\end{align}
which is a two-tone multisine input.
Since $\phi_k\in\BBR^6$ requires three distinct frequencies for $H$ in \eqref{eq:regressor_sinusoidal_mapping} to have full rank, the two-tone input $\bar u_k$ alone is not persistently exciting.

Three choices of $\sigma_k$ are considered: $\sigma_k=0$ (Case 1); $\sigma_k=0.5$, constant (Case 2); and $\sigma_k$ self-regulating per \eqref{eq:self_regulating_sigma} with $\sigma_0=0.5$, $\beta=0.8$ (Case 3).
Figure~\ref{fig:three_case} shows, for all three cases, the parameter error $\|\theta_k-\theta\|$, the covariance measure $\lambda_{\max}(P_k)$, the Fisher information measure $\lambda_{\min}(\mathcal F_k)$, and the excitation amplitude $\sigma_k$.

In Case 1, the parameter error does not converge to zero, $\lambda_{\max}(P_k)$ plateaus, and $\lambda_{\min}(\mathcal F_k)$ plateaus rather than growing without bound, confirming the lack of persistent excitation identified above.

In Cases 2 and 3, the injected noise restores persistent excitation: the parameter error and $\lambda_{\max}(P_k)$ converge to zero, and $\lambda_{\min}(\mathcal F_k)$ grows without bound, consistent with Theorem~\ref{thm:self_regulating}.
However, $\sigma_k$ remains constant in Case 2, whereas $\sigma_k$ vanishes in Case 3, confirming the remaining claim of Theorem~\ref{thm:self_regulating}.
\hfill$\Diamond$
\end{exmp}

\begin{figure}[h]
\centering
\includegraphics[width=\columnwidth]{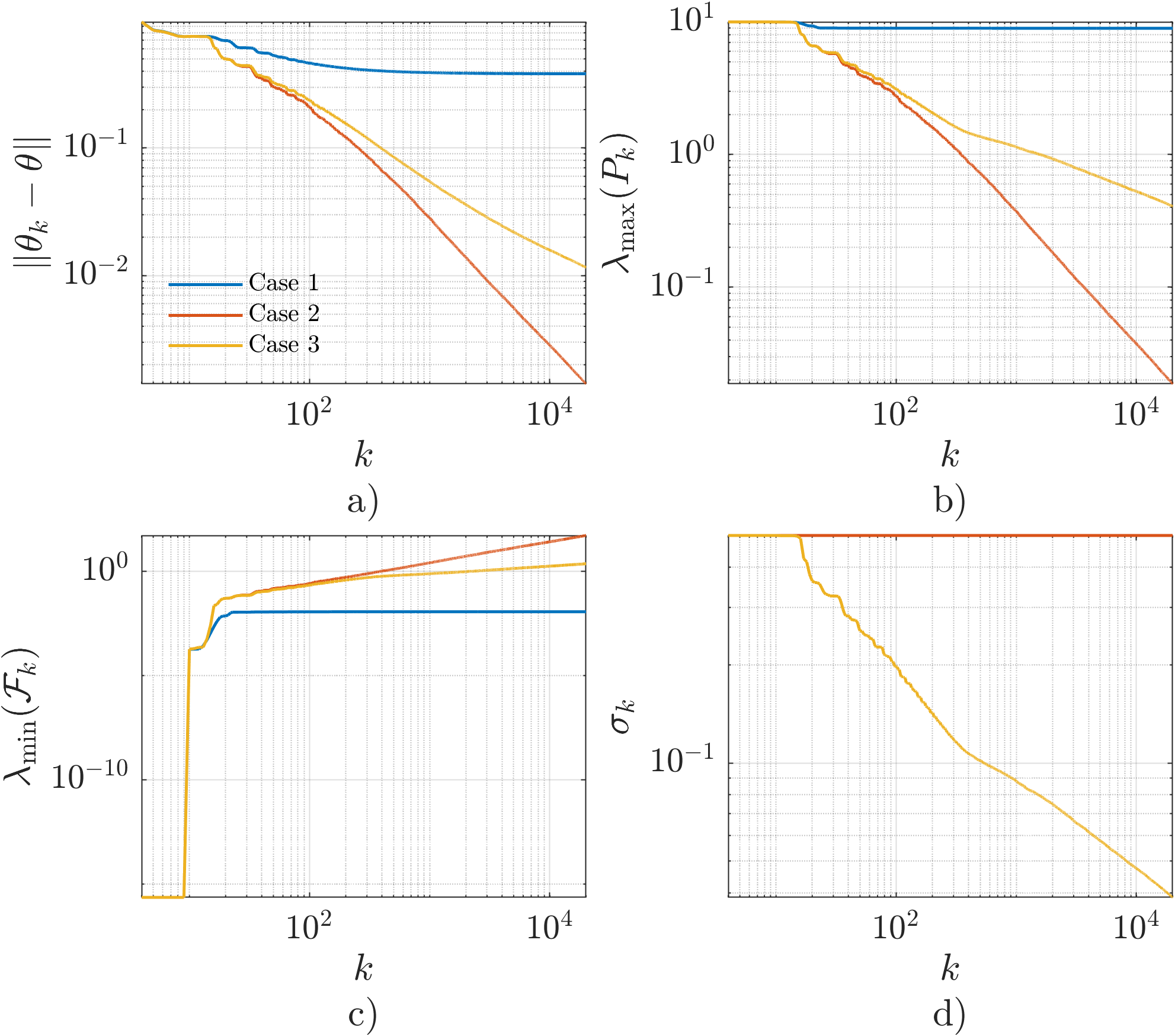}
\caption{Example~\ref{ex:three_case}. 
Effect of self-regulating excitation, confirming Theorem~\ref{thm:self_regulating}. 
a) shows the parameter error $\|\theta_k-\theta\|$,
b) shows the largest eigenvalue of the covariance matrix $\lambda_{\max}(P_k)$,
c) shows the smallest eigenvalue of the Fisher information matrix $\lambda_{\min}(\mathcal F_k)$, and 
d) shows the excitation amplitude $\sigma_k$. Case 1 (2 sines only) fails to converge and $\lambda_{\min}(\mathcal F_k)$ plateaus; Cases 2 and 3 both converge and $\lambda_{\min}(\mathcal F_k)$ grows without bound, but $\sigma_k$ vanishes only in Case 3.}
\label{fig:three_case}
\end{figure}
% \todo{figure titles need to be fixed}
\begin{exmp}
\label{ex:beta_sweep}
\textup{Effect of $\beta$ on the rate of covariance decay.}
Reconsider Example~\ref{ex:three_case} with $\sigma_k$ given by \eqref{eq:self_regulating_sigma} for
$$
\beta\in\{0,\,0.01,\,0.1,\,0.3,\,0.5,\,0.7,\,0.9,\,1.2,\,1.8\}.
$$
Figure~\ref{fig:beta_sweep} shows, for each value of $\beta$, $\lambda_{\max}(P_k)$, $\sigma_k$, and the ratio $\sigma_k/\lambda_{\max}(P_k)$, together with the exponent $\gamma$ obtained by fitting a line to $\log\lambda_{\max}(P_k)$ versus $\log k$ over the final $70\%$ of the run.
%
% For every $\beta$ in this sweep, $\sigma_k/\lambda_{\max}(P_k)\rightarrow\infty$, confirming the corresponding claim of Proposition~\ref{prop:self_regulating} across the full range of tested values, including $\beta=1.2$ and $\beta=1.8$, which fall outside the hypothesis $0<\beta<1$ of Proposition~\ref{prop:self_regulating}.
For $\beta<1$ in this sweep, $\sigma_k/\lambda_{\max}(P_k)\rightarrow\infty$, confirming the claim of Theorem~\ref{thm:self_regulating}.
For $\beta=1.2$ and $\beta=1.8$, which fall outside the hypothesis $0<\beta<1$ of Theorem~\ref{thm:self_regulating}, $\sigma_k/\lambda_{\max}(P_k)\rightarrow0$ instead, since $\sigma_k/p_k=(\sigma_0/p_0^\beta)p_k^{\beta-1}$ and $\beta-1>0$ reverses the sign of the limiting exponent as $p_k\rightarrow0$.
%
% Thus the hypothesis $\beta<1$ in Proposition~\ref{prop:self_regulating} is not a proof convenience: for $\beta>1$, the excitation amplitude decays faster than the covariance measure, the opposite of the behavior the proposition guarantees.
%
% The fitted rate $\gamma$ decreases monotonically as $\beta$ increases, confirming Remark~\ref{rem:beta_effect}: smaller $\beta$ yields faster covariance decay at the cost of more persistent excitation, larger $\beta$ yields the reverse, and no intermediate value of $\beta$ improves both simultaneously.
The fitted rate $\gamma$ decreases monotonically as $\beta$ increases, confirming Remark~\ref{rem:beta_effect}.
Smaller $\beta$ yields faster covariance decay at the cost of more persistent excitation, larger $\beta$ yields the reverse, and no intermediate value of $\beta$ improves both simultaneously.
\hfill$\Diamond$
\end{exmp}

\begin{figure}[t]
\centering
\includegraphics[width=\columnwidth]{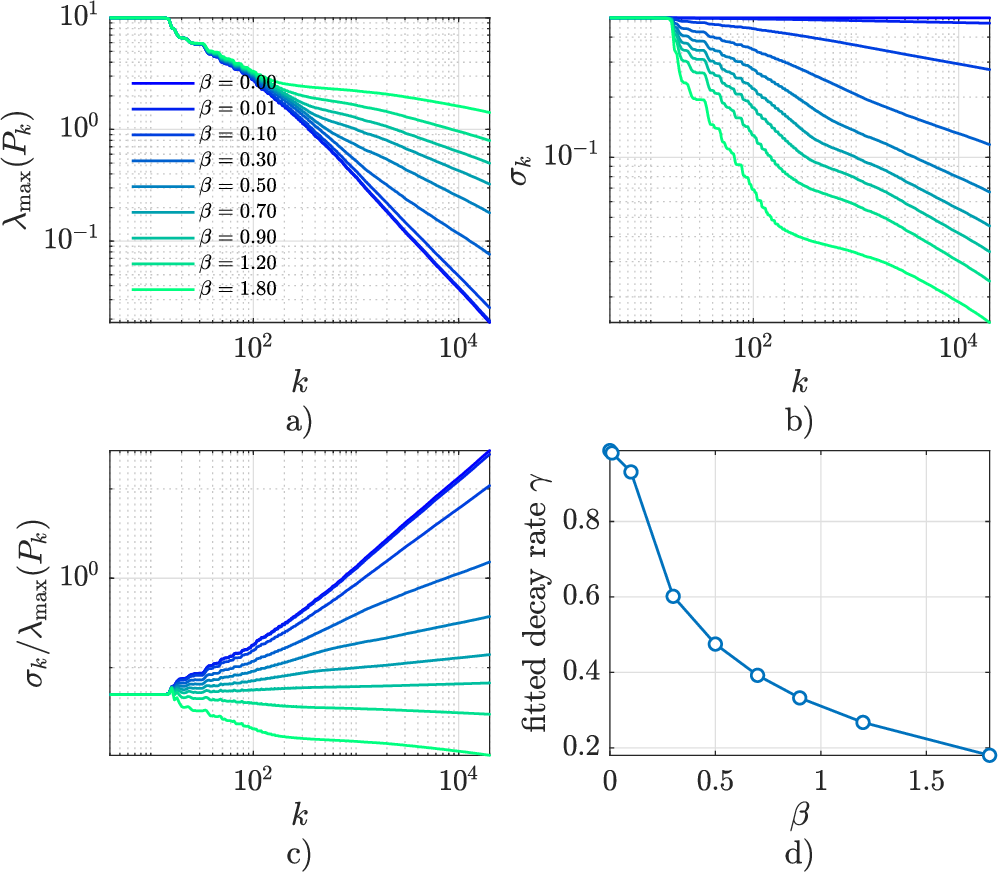}
\caption{Example~\ref{ex:beta_sweep}. Effect of $\beta$ on covariance decay, excitation amplitude, and their ratio. 
a) shows $\lambda_{\max}(P_k)$,
b) shows $\sigma_k$, 
c) shows $\sigma_k/\lambda_{\max}(P_k)$, and 
d) shows fitted decay rate $\gamma$ for various values of $\beta.$}
\label{fig:beta_sweep}
\end{figure}

% \section{Future Work}
% \begin{itemize}
%     \item Extension to vector-valued inputs and more general state-space models, which leads naturally to matrix RLS formulations.
%     %
%     \item Investigation of alternative self-regulating update laws for the excitation amplitude $\sigma_k$, including nonlinear and direction-dependent constructions.
%     %
%     \item Analysis of the effect of forgetting factors on the interaction between information loss, excitation-signal decay, and parameter convergence.
% \end{itemize}

\section{Conclusions}
\label{sec:conclusions}

This paper proposed a self-regulating excitation law that ties the probing amplitude to the recursive least-squares covariance matrix rather than to the prediction error.
%
% The excitation amplitude is persistently exciting in expectation for every nonzero value, and, under a mild information-accumulation condition, converges to zero while the parameter estimate converges to the true parameter, with the excitation vanishing more slowly than the covariance measure.
The excitation amplitude is persistently exciting in expectation for every nonzero value, and converges to zero while the parameter estimate converges to the true parameter, with the excitation vanishing more slowly than the covariance measure.
Numerical results confirmed both properties and showed that the law's exponent trades the rate of covariance decay against the amount of excitation retained.

Future work will analyze the effect of forgetting factors on the interaction between information loss, excitation-signal decay, and parameter convergence.

\appendix
\section{Proofs}
\subsection{Proof of Proposition \ref{prop:multisine_rank}}
\label{proof:multisine_rank}

\begin{proof}
Using \eqref{eq:regressor_sinusoidal_mapping},
\begin{align}
    \sum_{i=k}^{k+N-1}
        \phi_i\phi_i^\rmT
        &=
        H
        \left(
            \sum_{i=k}^{k+N-1}
                s_i s_i^\rmT
        \right)
        H^\rmT.
        \nn
\end{align}
By Lemma \ref{lem:sinusoidal_PE}, there exist an integer $N>0$ and a scalar
$\alpha_s>0$ such that
\begin{align}
    \sum_{i=k}^{k+N-1}
        s_i s_i^\rmT
        &\geq
        \alpha_s I_{2n}
        \nn
\end{align}
for all $k$.
Therefore,
\begin{align}
    \sum_{i=k}^{k+N-1}
        \phi_i\phi_i^\rmT
        &\geq
        \alpha_s H H^\rmT.
        \nn
\end{align}
Since $\rank H=2n$, the matrix $H$ is nonsingular, and thus
$$
H H^\rmT
    \geq
    \sigma_{\min}^2(H) I_{2n},
$$
where $\sigma_{\min}(H)>0$ is the minimum singular value of $H$.
Consequently,
\begin{align}
    \sum_{i=k}^{k+N-1}
        \phi_i\phi_i^\rmT
        &\geq
        \alpha_s
        \sigma_{\min}^2(H)
        I_{2n}.
        \nn
\end{align}
Thus, $\{\phi_k\}$ is persistently exciting.
\end{proof}

\subsection{Proof of Proposition \ref{prop:whitenoise_rank}}
\label{proof:whitenoise_rank}

\begin{proof}
In steady state,
\begin{align}
    \mathbb E
    \left[
        \sum_{i=k}^{k+N-1}
            \phi_i\phi_i^\rmT
    \right]
        &=
        \sum_{i=k}^{k+N-1}
            \mathbb E
            \left[
                \phi_i\phi_i^\rmT
            \right]
        =
        N R_\phi.
        \nn
\end{align}
Since $
R_\phi
    \geq
    \lambda_{\min}(R_\phi)I_{2n},
$
it follows that 
\begin{align}
    \mathbb E
    \left[
        \sum_{i=k}^{k+N-1}
            \phi_i\phi_i^\rmT
    \right]
        &\geq
        N\lambda_{\min}(R_\phi)I_{2n}
    \nn.
\end{align}
Thus, $\{\phi_k\}$ is persistently exciting in expectation.
\end{proof}

\subsection{Proof of Proposition \ref{prop:constant_sigma}}
\label{proof:constant_sigma}

\begin{proof}
By linearity of $G(q)$, $\phi_k=\bar\phi_k+\sigma\phi_k^w$, where $\bar\phi_k$ is generated by $\bar u_k$ and is deterministic.
Since $\phi_k^w$ is zero mean,
\begin{align}
    \mathbb E[\phi_k\phi_k^\rmT]
        &=
        \bar\phi_k\bar\phi_k^\rmT
        +
        \sigma^2R_w
        \succeq
        \sigma^2R_w.
    \nn
\end{align}
Summing over $i=k,\ldots,k+N-1$ and using $R_w\succeq\lambda_{\min}(R_w)I_{2n}$ gives (a).

By definition of $\SF_k^w$ and stationarity of $\phi_k^w$, $\mathbb E[\SF_k^w]=k\sigma^2R_w$.
Since $\lambda_{\min}(A+B)\geq\lambda_{\min}(A)+\lambda_{\min}(B)$ for symmetric $A$ and $B$,
\begin{align}
    \lambda_{\max}(\bar P_k)
        &=
        \frac{1}{\lambda_{\min}\left(P_0^{-1}+k\sigma^2R_w\right)}
        \nn \\
        &
        \leq
        \frac{1}{\lambda_{\min}(P_0^{-1})+k\sigma^2\lambda_{\min}(R_w)},
    \nn
\end{align}
which gives (b).
\end{proof}

\subsection{Proof of Proposition \ref{prop:information_bound}}
\label{proof:information_bound}

\begin{proof}
Since $p_k$ is nonincreasing, $\sigma_k\leq\sigma_0$.
Stability of $G(q)$ and boundedness of $\bar u_k$, $\sigma_k$, and $w_k$ give $\|\phi_k\|^2\leq B$.

Fix $k\geq1$ and $k\leq l\leq k+N-1$.
Then $P_l^{-1}\preceq P_{k-1}^{-1}+NB\,I_{2n}\preceq(1+NBp_0)P_{k-1}^{-1}$, so $p_l\geq p_{k-1}/(1+NBp_0)$ and, by \eqref{eq:self_regulating_sigma},
\begin{align}
    \sigma_l
        &\geq
        \rho\,\sigma_{k-1},
    \qquad
    \rho
        \isdef
        (1+NBp_0)^{-\beta}.
    \nn
\end{align}

For a unit vector $v$ and $k\leq i\leq k+N-1$, linearity of $G(q)$ gives
\begin{align}
    v^\rmT\phi_i
        &=
        m_i
        +
        \sum_{l=k}^{i-1}
            c_{i-l}(v)\,\sigma_l w_l,
    \nn
\end{align}
where $m_i$ is computable from $\SH_{k-1}$ and $c_d(v)$ are fixed coefficients determined by $v$ and the impulse response of $G(q)$.
Since $\sigma_l$ is computable from $\SH_{l-1}$ and $w_l$ is independent of $\SH_{l-1}$ with zero mean and unit variance, the cross terms vanish under $\BBE[\cdot\mid\SH_{k-1}]$, and
\begin{align}
    v^\rmT
    \BBE\left[\SF_{k+N}-\SF_k\mid\SH_{k-1}\right]
    v
        &\geq
        \rho^2\sigma_{k-1}^2\sum_{d=1}^{N-1}(N-d)\,c_d(v)^2.
    \nn
\end{align}

For minimal $G(q)$ and $N=2n+1$, no unit $v$ has $c_1(v)=\cdots=c_{2n}(v)=0$, since this would give a pole-zero cancellation between the $v$-weighted regressor channels and $G(q)$.
The sum is continuous in $v$, so its minimum $r$ over the unit sphere is positive.
Taking $c=\rho^2r$ completes the proof.
\end{proof}

\subsection{Proof of Theorem \ref{thm:self_regulating}}
\label{proof:self_regulating}

\begin{proof}
Since $P_k^{-1}=P_0^{-1}+\SF_k$ and $\SF_k$ is nondecreasing in the Loewner order, $p_k$ is nonincreasing and converges to some $p_\infty\geq0$ on every sample path.

Consider a sample path with $p_\infty>0$.
Then $\sigma_k\rightarrow\bar\sigma\isdef\sigma_0(p_\infty/p_0)^\beta>0$, so there exists $K\geq1$ such that $\sigma_{k-1}^2\geq\bar\sigma^2/2$ for all $k\geq K$.
Define $t_j\isdef K+jN$, $\SG_j\isdef\SH_{t_j-1}$, and $\Delta_j\isdef\SF_{t_{j+1}}-\SF_{t_j}$.
By Proposition~\ref{prop:information_bound} with $k=t_j$, $\BBE[\Delta_j\mid\SG_j]\succeq(c\bar\sigma^2/2)I_{2n}$.
Since $\|\phi_i\|^2\leq B$, $\|\Delta_j\|\leq NB$.
Hence $D_j\isdef\Delta_j-\BBE[\Delta_j\mid\SG_j]$ is a bounded martingale difference sequence, and the strong law of large numbers for martingales gives $S_n/n\rightarrow0$ almost surely, where $S_n\isdef\sum_{j=0}^{n-1}D_j$.
Then
\begin{align}
    \lambda_{\min}\left(\frac{\SF_{t_n}}{n}\right)
        &\geq
        \frac{c\bar\sigma^2}{2}
        -
        \left\|\frac{S_n}{n}\right\|
        -
        \frac{\|\SF_{t_0}\|}{n},
    \nn
\end{align}
so $\liminf_n\lambda_{\min}(\SF_{t_n})/n\geq c\bar\sigma^2/2>0$.
Since $\lambda_{\min}(\SF_k)$ is nondecreasing, $\lambda_{\min}(\SF_k)\rightarrow\infty$ almost surely on this event.
This gives $p_k\leq\left(\lambda_{\min}(P_0^{-1})+\lambda_{\min}(\SF_k)\right)^{-1}\rightarrow0$, which contradicts $p_\infty>0$.
Therefore $p_\infty=0$ almost surely.

It follows that $\sigma_k\rightarrow0$ and $p_k\rightarrow0$ almost surely.
Also $\lambda_{\min}(\SF_k)\geq p_k^{-1}-\lambda_{\max}(P_0^{-1})\rightarrow\infty$, and
\begin{align}
    \frac{\sigma_k}{p_k}
        &=
        \frac{\sigma_0}{p_0^\beta}\,p_k^{\beta-1}
        \rightarrow
        \infty
    \nn
\end{align}
since $0<\beta<1$.

Finally, \eqref{eq:rls_theta}--\eqref{eq:rls_P} with $y_k=\phi_k^\rmT\theta$ give $\hat\theta_k-\theta=P_kP_0^{-1}(\hat\theta_0-\theta)$, so $\|\hat\theta_k-\theta\|\leq p_k\|P_0^{-1}\|\|\hat\theta_0-\theta\|\rightarrow0$ almost surely.
\end{proof}

\printbibliography

@book{pintelon2012system,
  title={System identification: a frequency domain approach},
  author={Pintelon, Rik and Schoukens, Johan},
  year={2012},
  publisher={John Wiley \& Sons}
}

@ARTICLE{goelrecursive2020,
  author={Goel, Ankit and Bruce, Adam L. and Bernstein, Dennis S.},
  journal={IEEE Control Systems Magazine}, 
  title={Recursive Least Squares With Variable-Direction Forgetting: Compensating for the Loss of Persistency [Lecture Notes]}, 
  year={2020},
  volume={40},
  number={4},
  pages={80-102},
  doi={10.1109/MCS.2020.2990516}}

@book{astrom1995adaptive,
  title={Adaptive Control},
  author={{\AA}str{\"o}m, Karl J. and Wittenmark, Bj{\"o}rn},
  edition={2nd},
  year={1995},
  publisher={Addison-Wesley},
  address={Reading, MA},
  isbn={0-201-55866-1}
}

@inproceedings{adetola2006excitation,
  title={Excitation Signal Design for Parameter Convergence in Adaptive Control of Linearizable Systems},
  author={Adetola, Veronica and Guay, Martin},
  booktitle={Proceedings of the 45th IEEE Conference on Decision and Control},
  pages={447--452},
  year={2006},
  doi={10.1109/CDC.2006.377322}
}

@article{cao2007intelligent,
  title={Intelligent Excitation for Adaptive Control with Unknown Parameters in Reference Input},
  author={Cao, Chengyu and Hovakimyan, Naira and Wang, Jiang},
  journal={IEEE Transactions on Automatic Control},
  volume={52},
  number={8},
  pages={1525--1532},
  year={2007},
  doi={10.1109/TAC.2007.902780}
}
\end{document}